\documentclass[12pt,letterpaper,reqno]{amsart}

\usepackage{times}
\usepackage[T1]{fontenc}
\usepackage{mathrsfs}
\usepackage{latexsym}
\usepackage[dvips]{graphics}
\usepackage{epsfig}
\usepackage{amsmath,amsfonts,amsthm,amssymb,amscd}
\input amssym.def
\input amssym.tex

\newcommand\be{\begin{equation}}
\newcommand\ee{\end{equation}}
\newcommand\bea{\begin{eqnarray}}
\newcommand\eea{\end{eqnarray}}
\newcommand\bi{\begin{itemize}}
\newcommand\ei{\end{itemize}}
\newcommand\ben{\begin{enumerate}}
\newcommand\een{\end{enumerate}}
\newcommand\bc{\begin{center}}
\newcommand\ec{\end{center}}
\newcommand\ba{\begin{array}}
\newcommand\ea{\end{array}}

\newcommand{\bd}{{\bf d}}

\newcommand{\C}{\ensuremath{\mathbb{C}}}

\newcommand{\sS}{{\mathcal S}}

\newcommand{\sA}{{\mathcal A}}
\newcommand{\sD}{{\mathcal D}}
\newcommand{\bu}{{\bf u}}
\newcommand{\bmu}{{\bf \mu}}

\newcommand{\sT}{{T}}
\newcommand{\CC}{{\mathbb C}}
\newcommand{\ZZ}{{\mathbb Z}}

\newtheorem{thm}{Theorem}[section]

\newtheorem{cor}[thm]{Corollary}
\newtheorem{lem}[thm]{Lemma}
\newtheorem{prop}[thm]{Proposition}

\newtheorem{exa}[thm]{Example}
\newtheorem{defi}[thm]{Definition}

\theoremstyle{definition}
\newtheorem{rem}[thm]{Remark}

\numberwithin{equation}{section}

\begin{document}

\title{ Notes on Cardinal's  
 Matrices}

\author{Jeffrey C.  Lagarias}
\address{Dept. of Mathematics\\
University of Michigan \\
Ann Arbor, MI 48109-1043\\
}
\email{lagarias@umich.edu}

\author{David Montague}
\address{Dept. of Mathematics\\
Stanford University\\
Stanford, CA 94305
}
\email{davmont@stanford.edu}


\thanks{The first author's work  was partially supported by NSF grants DMS-0801029
DMS-1101373 and  DMS-1401224. The second author's work was supported
by the NSF as part of an REU project.}

\date{November 21, 2015, v24}

\maketitle

\begin{abstract}These notes are motivated by the work of Jean-Paul Cardinal  on
certain symmetric matrices related to the Mertens function.  He showed
that certain norm bounds
on his matrices implied the Riemann hypothesis. 
Using a different matrix norm we show an equivalence of
the Riemann hypothesis to suitable norm bounds
on his Mertens matrices in the new norm. 
We also specify a deformed
version of his Mertens matrices that unconditionally satisfy a norm bound  of 
the same strength as this Riemann hypothesis bound.  
\end{abstract}
\setcounter{equation}{0}


%
%
\section{Introduction}

In 2010 Jean-Paul Cardinal \cite{C10E} introduced for each $n \ge 1$  two  symmetric integer matrices
 $\mathcal{U}_n$ and $\mathcal{M}_n$  constructed using
a set of ``approximate divisors" of $n$, defined below. Each of these matrices can be obtained from the other. 
 Both matrices are $s \times s$ matrices, where $s=s(n)$
is about $2 \sqrt{n}$. Here  $s(n)$ counts the number of  distinct values taken by $\lfloor \frac{n}{k} \rfloor$ 
when $k$ runs from $1$ to $n$. These values comprise the approximate divisors, and we  label them
$k_1 =1,k_2, ..., k_s=n$ in increasing order.

We start with defining Cardinal's basic matrix  $\mathcal{U}_n$ attached to an integer $n$.
The $(i, j)$-th entry of $\mathcal{U}_n$ is  $\lfloor \frac{n}{k_i k_j} \rfloor$. One can show that
such matrices take value  $1$ on the anti-diagonal, and are $0$
at all entries below the antidiagonal, consequently it has determinant $\pm1$.
(The antidiagonal entries are the $(i, j)$-entries  with $i+j= s+1$, numbering rows and
columns from $1$ to $s$.) All entries above the antidiagonal are positive.
It follows that this matrix has $s^2 \approx 4n$ entries and
about half of them are nonzero.
The matrices $\mathcal{U}_n$ encode information mixing together the additive
and multiplicative structures of the integers in an interesting way, using also the floor function.

Cardinal  defines a second matrix,  $\mathcal{M}_n$,   by the recipe
$$\mathcal{M}_n := \sT (\mathcal{U}_n)^{-1}  \sT,$$
where $\sT$ is a square matrix of the same size,   having $1$'s on and above the anti-diagonal, and 
value $0$ at all entries strictly below the antidiagonal.
Clearly $\det(\mathcal{M}_n) = \det(\mathcal{U}_n)$.
One can prove that $\mathcal{M}_n$ has a similar pattern of entries to $\mathcal{U}_n$, 
in having all values $0$ strictly below the antidiagonal and all values $1$ on the antidiagonal, but it
may  now have some negative entries above the antidiagonal.  One of Cardinal's main results is that
the entries above the antidiagonal its $(i, j)$-th entry is
$M( \lfloor \frac{n}{k_i k_j}\rfloor)$, where $M(x) = \sum_{j \le x} \mu(j)$
is the Mertens function of prime number theory. The entries below the antidiagonal
are also Mertens function values, since in  that case $M(\lfloor \frac{n}{k_i k_j}\rfloor)= M(0)=0$.
We might therefore name this matrix a {\em Mertens matrix}.

Cardinal \cite[Theorem 24]{C10E} proves that
an upper bound $O( n^{\frac{1}{2} + \epsilon})$
on the  growth   of the norms of the matrices  $\mathcal{M}_n$ measured  
in the $\ell_2$ operator norm implies the Riemann hypothesis holds. 
He gives numerical plots of the norms of  $\mathcal{M}_n$ for small $n$ supporting this upper bound.

 Cardinal's results  are  structural. He relates these matrices
 to finite-dimensional
quotient  algebras of the  algebra of Dirichlet series, and proves that  his quotient algebras
are commutative and associative matrix algebras, which are lower triangular.
His  Proposition \ref{pr22} about floor functions  is important in establishing that
certain linearly transformed versions of the matrices in the algebra give  
rise to symmetric matrices, including $\mathcal{U}_n$ and $\mathcal{M}_n$.

In the 2008 French preprint version \cite{C10F} of this paper Cardinal 
 introduced additionally  a deformed 
version $\widetilde{\mathcal{U}}_n^{+}$ of his matrix $\mathcal{U}_n$,
whose entries when rounded down by the floor function yield $\mathcal{U}_n$.
He then proposed to define by the same recipe a deformed version of 
the Mertens matrix $\mathcal{M}_n$ as
$$
\widetilde{\mathcal{M}}_n^{+} := T_{s} (\tilde{\mathcal{U}_n^{+}})^{-1} T_{s}.
$$
He presented an argument giving a matrix norm for the perturbed matrix, asserting that it satisfied the Riemann hypothesis
bound.  Unfortunately his   argument failed  because his 
 definition
of a deformed matrix $\widetilde{\mathcal{U}}_n^{+}$ had   rank one,
so was not invertible. Perhaps because of this
no discussion of  deformed matrices  appeared in  the author's final English version of his paper.
Cardinal's argument, though flawed, has serious content, and we obtain below a 
modified result where it works.

\subsection{Results}\label{sec11}

Much  of this note consists of an exposition of Cardinal's results in a 
very slightly different notation, presenting similar numerical examples.
The main new contributions are the following.
\begin{enumerate}
\item
We use a different family of matrix norms, the Frobenius norms. 
We  prove that the Riemann hypothesis  is equivalent to a suitable growth bound
on the Frobenius matrix norms $||\mathcal{M}_n||_F$ (Theorem \ref{th43}).
The choice of the family of norms possibly matters to obtain an equivalence
because the dimensionality of the matrices goes
to infinity as $n \to \infty$.  

\item
 We introduce in  Section \ref{sec62} a natural modified definition of 
deformed matrix $\widetilde{\mathcal{U}}_n$, that applies the floor
function to ``half" of Cardinal's deformed  matrix,  which is  entrywise close to $\mathcal{U}_n$.
We set $\widetilde{\mathcal{M}}_n :=  T(\widetilde{\mathcal{U}}_n)^{-1} T$ and
prove for $\widetilde{\mathcal{M}}_n$
unconditionally a Frobenius  norm upper bound of 
the same strength as the Riemann hypothesis bound above (Corollary \ref{cor512}).
This proof  follows the ideas given in Cardinal's
2008 French preprint \cite{C10F}.
\end{enumerate}


\subsection{Related work} \label{sec12}
There is  earlier work on 
integer matrices having a determinant related to the Mertens function 
(\cite{BFP88}, \cite{RB89}, \cite{BJ92}, \cite{Va93}, \cite{Va96}, \cite{Hu97}).
This work concerns {\em Redheffer's matrix}, 
named after Redheffer \cite{Red77}. It is an interesting question to
determine if there are relations between Redheffer's matrix  and Cardinal's matrix.
Recent work of Cardon \cite{Car10} makes a connection of the 
Redheffer matrix with the  quantities $\lfloor \frac{n}{k}\rfloor$ appearing in Cardinal's
matrix.\\

\subsection{Matrix norms}\label{sec13}
In this paper we bound matrices using the Frobenius matrix norm
$$
||M||_{F}^2 := \sum_{i=1}^n \sum_{j=1}^n |M_{ij}|^2.
$$

In Cardinal's paper \cite{C10E} the matrix norm $||M||$ 
on an $n \times n$ complex matrix is 
taken to be the {\em $l_2$ operator norm}  
$$
||M|| := \max_{ v \in \C^n, v \ne 0} \frac{ ||Mv||^2}{||v||^2},
$$
where $||v||^2 = \sum_{i=1}^n  |v_i|^2$. 
For symmetric real matrices $M$ this norm bound coincides with the
spectral radius of the matrix.

%
%
\section{Cardinal's Algebra of  Approximate Divisors}\label{sec2}

Cardinal's paper is concerned with  matrices encoding properties
of ``approximate divisors" of $n$. The surprising property they have
is of forming a commutative subalgebra $\sA_n$ of lower triangular matrices,
of rank $s= s(n)$.

These matrices encode information mixing together the additive
and multiplicative structures of the integers in an interesting way.\smallskip
 
%
%
\subsection{Approximate Divisors}\label{sec21}

First, Cardinal introduces  ``approximate divisors" of $n$ (our terminology). The number of
such  divisors is on the order of twice the square root of $n$.

\begin{defi} \label{de21}
Let $\mathcal{S}_n$ 
be the set of distinct integers
of the form $\{\lfloor \frac{n}{k} \rfloor: \, 1 \le k \le n\}$.
Let
$$ s = s(n) := \#(\mathcal{S}_n),$$
 so that $s(n)= 2\lfloor \sqrt{n}\rfloor,$
or $\, 2\lfloor \sqrt{n}\rfloor -1$.
\end{defi}

The set $\mathcal{S}_n=\mathcal{S}_n^{-} \bigcup \mathcal{S}_n^{+}$
 consists of all integers in 
$$
\mathcal{S}_n^{-} :=\{ j:~1 \le j \le \lfloor \sqrt{n}\rfloor\},
$$
together with the complementary set 
$$
\mathcal{S}_n^{+} :=
\{ \lfloor \frac{n}{j} \rfloor:~1 \le j \le \lfloor \sqrt{n} \rfloor \}.
$$
These sets are disjoint if  $m(m+1) \le n < (m+1)^2$ and they 
have exactly one element in common, namely $m= \lfloor \sqrt{n} \rfloor$,
if $m^2 \le n < m(m+1)$. He shows \cite[Prop. 4]{C10E}:

\begin{lem}\label{le22} {\em (Cardinal)}
Number the elements of $\mathcal{S}$ in increasing order as $k_i$, $1 \le i \le s$,
so $k_1=1, k_s=n$. 
The map $~~\widehat{}: \sS \to \sS$ defined by 
$$
\widehat{k_i} :=  \lfloor \frac{n}{k_i} \rfloor =k_{s+1-i},
 $$
is an involution exchanging $\sS_n^{+}$ and $\sS_n^{-}.$ 
\end{lem}

\begin{exa}\label{ex23a}
{\rm
For $n=16$ the approximate divisors are $\{1, 2, 3, 4, 5, 8, 16\}$.
The involution acts $\widehat{1}=16, ~\widehat{16}=1$, and $\widehat{2}=8, ~\widehat{8}=2$,
and $\widehat{3}=5, ~\widehat{5}=3$ and $\widehat{4}=4$.
}
\end{exa}

%
%
\subsection{Cardinal's Multiplication Algebra $\sA_n$}\label{sec22}

Cardinal associates to an integer $n$  a commutative, associative
algebra $\sA_n$ of lower triangular matrices of dimension $s$ inside the $s \times s$ matrices,
for which he gives generators.
The generators give the effect of ``multiplication by a fixed divisor" on the this algebra.

\begin{thm}\label{th21} {{\em (Cardinal)}}
The algebra $\sA_n$ has rank $s$ and is spanned  by the 
$s \times s$ matrices $\}\rho_n(k); k \in \sS_n$, 
 where
$k$ runs over the set of approximate divisors of $n$, which are all integers of
the form $\lfloor \frac{n}{k} \rfloor$. Each  matrix $\rho_n(k)$ is a lower
triangular matrix giving the effect of multiplication by $k$ on a basis of 
approximate divisors, arranged in increasing order. The matrix
$\rho(1)$ is the identity matrix, and all other $\rho_n(k)$ are  lower triangular nilpotent.
This algebra is commutative.
\end{thm}

The commutativity property is a consequence of the  identities in 
Proposition \ref{pr22}.

\begin{exa}\label{ex25a}
{\rm
For $n=16$ the 
multiplication by k matrices ${\rho (k)}$ for $k=2, 3, 4,5$
are given below
(omitted  entries are $0$). All these matrices are nilpotent.
\[ \rho_{16}(2) = 
  \begin{tabular}{ |c | c c c c c c r | }
  \hline
    ${\rho_{16}(2)}$ & 1 & 2 & 3 & 4 & 5 & 8 & 16 \\
    \hline
    1  &    &   &  &    &  & &  \\
    2  & 1    &    &    &    &   &  &   \\
    3  &   &  &    &  &  &   &   \\
    4  &     & 1  &  &   &       &  &   \\
    5  &      &  &   &     &       &   &   \\
    8  &    &    &    1   &  1   &       &    &   \\
    16 &     &     &       &     &  1     & 1  &   \\
    \hline
  \end{tabular}
\]

\[ \rho_{16}(3)=
  \begin{tabular}{ |c | c c c c c c r | }
  \hline
    ${\rho_{16}(3)}$ & 1 & 2 & 3 & 4 & 5 & 8 & 16 \\
    \hline
    1  &    &   &  &    &  & &  \\
    2  &     &    &    &    &   &  &   \\
    3  &  1 &  &    &  &  &   &   \\
    4  &     &   &  &   &       &  &   \\
    5  &      &  &   &     &       &   &   \\
    8  &    &   1 &       &     &       &    &   \\
    16 &     &     &     1  &    1 &  1     &   &   \\
    \hline
  \end{tabular}
\]
\[ \rho_{16}(4)= 
  \begin{tabular}{ |c| c c c c c c r | }
  \hline
    ${\rho_{16}(4)}$ & 1 & 2 & 3 & 4 & 5 & 8 & 16 \\
    \hline
    1  &    &   &  &    &  & &  \\
    2  &     &    &    &    &   &  &   \\
    3  &   &  &    &  &  &   &   \\
    4  &    1 &   &  &   &       &  &   \\
    5  &     &  &   &     &       &   &   \\
    8  &    &   1 &       &     &       &    &   \\
    16 &     &     &     1  &    1 &      &   &   \\
    \hline
  \end{tabular}
\]

\[ \rho_{16}(5)=
  \begin{tabular}{ |c | c c c c c c r | }
  \hline
    ${\rho_{16}(5)}$ & 1 & 2 & 3 & 4 & 5 & 8 & 16 \\
    \hline
    1  &    &   &  &    &  & &  \\
    2  &     &    &    &    &   &  &   \\
    3  &   &  &    &  &  &   &   \\
    4  &     &   &  &   &       &  &   \\
    5  &    1  &  &   &     &       &   &   \\
    8  &    &    &       &     &       &    &   \\
    16 &     &   1  &     1  &     &      &   &   \\
    \hline
  \end{tabular}
\]

}
\end{exa}

\begin{rem}\label{rem25}
The algebra $\sA_n$  is not  semi-simple.  In fact  all its generators,
aside from the identity element,  are nilpotent.  It has dimension exactly $s$,
equal to its size. (Commutative subalgebras of matrix algebras over $\CC$ can have
larger dimension than their size $s$ for $s \ge 6$. The maximal dimension is 
$\lfloor \frac{1}{4}s^2\rfloor +1$, a result found by  I. Schur \cite{Sec05}. A simple proof was given by
Mirzakhani \cite{Mir98}).
\end{rem}

%
%
\subsection{Dirichlet Series and Cardinal's Multiplication Algebra $\sA_n$}\label{sec23}

Cardinal showed the  algebra $\sA_n$   to be 
a homomorphic image of  the algebra of formal
Dirichlet series.
Let $\sD_{\ZZ}$ denote the $\mathbb{Z}$-algebra   of all (formal) Dirichlet series 
$$
D(s) =\sum_{n=1}^{\infty} a_n n^{-s}
$$
with integer coefficients $a_n \in \ZZ$.
Here  addition  on Dirichlet series coefficients is
pointwise  and multiplication is Dirichlet convolution,
see Tenenbaum \cite[Sect. 2.3]{Ten15}.
We describe a (formal) Dirichlet series by the row vector $\bu= (a_1, a_2, a_3 , \cdots)$
of its coefficients. The invertible elements in $\sD_{\ZZ}$ are those with $a_1 = \pm 1.$

If we  restrict to  the subalgebra 
$\sD_{\ZZ}^{c}$ of Dirichlet series that absolutely converge to a function $f(s)$ on some right half-plane, then
viewed as such functions on a half-plane, these algebra operations correspond
to pointwise addition and multiplication of the functions  in their joint convergence domain.

For an integer $1 \le j<n$, 
 we let the {\em minus operator} give to each $k$  the value $k^{-}$ that is its predecessor of $\sS_n$,
 and we artificially define $1^{-} := 0$.
If $1 \le k< \lfloor \sqrt{n} \rfloor$ then $k^{-} = k-1$.

Cardinal's analysis implies the following result \cite[Propositions 9 to 13.]{Car10}. 
\begin{thm}\label{th22} {\em (Cardinal)}
The map $\tilde{\rho}_n: \sD_{\ZZ} \to \sA_n$
given by
$$
\sum_{k \ge 1} a_k k^{-s} \mapsto \sum_{m \in \sS_n } \big(\sum_{m^{-} <   k \le m }  a_k\big) \rho_n(m)
$$
is an algebra homomorphism. 
\end{thm}

The kernels $\ker(\tilde{\rho}_n)$ of these maps seem worthy of further study, but
we will not treat them here.
\subsection{Matrix Image of Dirichlet series $\zeta(s)$ and $\frac{1}{\zeta(s)}$}\label{sec32}
 We  look at the image of
the Riemann zeta function $\zeta(s) = \sum_{k=1}^{\infty} k^{-s}$ and its 
inverse $1/\zeta(s) = \sum_{k=1}^{\infty} \mu(k) k^{-s}$ under this homomorphism.
 
 View $n$ as fixed.
Define the vector $\bu :=\bu_{n}$  to be a $1 \times s(n)$ row vector
that has   entries that  sum the function $1$ over those integers 
in the half-open intervals $(m^{-}, m]$ for $m \in \sS_n$.
The lower triangular matrix  $Z_n :=\rho_n(\bu_n)$ is then a finite analogue of the Riemann zeta function. 
Cardinal  also introduces a vector $\bmu := \bmu_n$ 
whose entries sum  the M\"{o}bius function over the interval $(k^{-}, k]$; the corresponding inverse
matrix $Z_n^{-1} := \rho_n(\bmu_n)$ is then a finite analogue of the  inverse of the Riemann zeta function.

\begin{exa}\label{ex26a}
{\rm
For $n=16$ the 
vector $\bu_{16}= (1, 1, 1, 1, 1, 3, 8)$.
The matrix $Z_{n}$ for ${\rho_{16}(\bu)}$ is
(omitting $0$ entries)
\[ Z_{16}  = 
  \begin{tabular}{ |c | c c c c c c r | }
  \hline
    $\rho_{16}(\bu)$ & 1 & 2 & 3 & 4 & 5 & 8 & 16 \\
    \hline
    1  &  1 &   &  &    &  & &  \\
    2  & 1    &  1  &    &    &   &  &   \\
    3  &  1 & 0 &   1 & &  &   &   \\
    4  &    1 & 1   & 0  &   1   & &  &   \\
    5  &    1  &  0&   0&   0  &  1    &   &   \\
    8  &   3 &   2 &    1   &  1   &   0    &  1  &   \\
    16 &  8   &  4   &    3   &  2   &  2   & 1  &  1 \\
    \hline
  \end{tabular}
\]
}
\end{exa}

\begin{exa}\label{ex27a}
{\rm
For $n=16$ the inverse (M\"{o}bius) vector is $\bmu=( 1, -1,-1, 0, -1, 0, 1)$.
The matrix $Z_n^{-1} = \rho_{16}(\bmu)$ is 
\[ (Z_{16})^{-1} =
  \begin{tabular}{ |c | c c c c c c r | }
  \hline
    $ \rho_{16}(\bmu)$ & 1 & 2 & 3 & 4 & 5 & 8 & 16 \\
    \hline
    1  &  1 &   &  &    &  & &  \\
    2  & -1    &  1  &    &    &   &  &   \\
    3  &  -1 & 0 &   1 &  &  &   &   \\
    4  &    0 & -1  &  0& 1  &       &  &   \\
    5  &    -1  &  0&   0&   0  &  1     &   &   \\
    8  &   0 &   -1 &    -1   &  -1   &   0    &  1  &   \\
    16 &  1   &  -1   &   - 2   &  -1   &  -2     & -1  & 1  \\
    \hline
  \end{tabular}
\]
}
\end{exa}

\section{Cardinal's Algebra of Symmetric Matrices}

\subsection{Cardinal's Algebra has symmetric matrix image}

A key observation of Cardinal concerns the  image of the algebra
$\sA_n$  under
left multiplication by a matrix $T$ having ones above and on its antidiagonal $a_{i, s+1-i}$,
and zeros below the antidiagonal.  An example of  $T$ for $s=7$  is given in example \ref{ex30a}.

Cardinal  showed  (\cite[Proposition 20]{C10E}) the following result.

\begin{thm}\label{th31}
 {(Cardinal)}
The linear map  $T: \sA_n \to Mat_{s(n) \times s(n)}$ given
by $A \mapsto TA$ has image in the set
of real symmetric matrices $Sym_{s(n) \times s(n)}.$
\end{thm}

To put this result in perspective, 
when $T$  multiplies on the left  a general lower triangular matrix $L$,
the resulting matrix $TL$ is in general {\em not} symmetric. This  symmetry property   is
a special property of elements of the algebra $\sA_n$. This symmetry property encodes the involution
$k \mapsto \widehat{k}$, which in turn encodes the identity in Proposition \ref{pr22}.

Note that this map $A \mapsto TA$  is not an algebra homomorphism, it preserves addition
but it  does not respect matrix multiplication  in general, i.e. usually  $T$ does not 
commute with all  the elements of $\sA_n$.

\begin{exa}\label{ex30a}
{\rm 
For $n=16$,  we have $s(n)= 7$ and the matrix  $T$  is
(omitted entries are $0$). The borders giving the values of
the approximate divisors are not part of the matrix.

\[  T=
  \begin{tabular}{ |c | c c c c c c r | }
  \hline
    $T$ & 1 & 2 & 3 & 4 & 5 & 8 & 16 \\
    \hline
    1  & 1   & 1  & 1 & 1   & 1  & 1 & 1 \\
    2  & 1    & 1   & 1   & 1  &  1  & 1 &   \\
    3  & 1   & 1 &  1  & 1 &  1&   &  \\
    4  & 1   & 1  &  1 & 1  &       &   &   \\
    5  & 1     &1  &   1&     &       &  &   \\
    8  & 1   &   1 &      &     &      &    &   \\
    16 &  1  &     &       &     &      &   &   \\
    \hline
  \end{tabular}
\]
}
\end{exa}

\begin{exa}\label{ex31a}
{\rm 
For $n=16$, multiplication of the matrix $\rho(2)$ by $T$  yields
a symmetric matrix $T \rho_{16}(2)$
(omitted entries are $0$).
\[ T \rho_{16}(2) =
  \begin{tabular}{ |c | c c c c c c r | }
  \hline      $T\rho_{16}(2)$ & 1 & 2 & 3 & 4 & 5 & 8 & 16 \\
    \hline
    1  & 1   & 1  & 1 & 1   & 1  & 1 &  \\
    2  & 1    & 1   & 1   & 1  &    &  &   \\
    3  & 1   & 1 &    &  &  &   &  \\
    4  & 1   & 1  &   &   &       &   &   \\
    5  & 1     &  &   &     &       &  &   \\
    8  & 1   &    &      &     &      &    &   \\
    16 &    &     &       &     &      &   &   \\
    \hline
  \end{tabular}
\]
}
\end{exa}

\subsection{Cardinal's $\mathcal{U}$-matrices}

Cardinal  now introduces an $s \times s$ symmetric matrix  $\mathcal{U}_n$
corresponding to the zeta vector $\bu$:
$$
{\mathcal U}_n := T_s Z_n = T_s \rho_n(\bu).
$$
The following characterization of this matrix  shows
that  it encodes the multiplication  for the approximate divisors in a manifestly
symmetric form, which we stated as the definition in the introduction of
the paper.

\begin{lem}\label{le22b}
{\em (Cardinal)} 
Let $K = \rm{diag}(k_1,...,k_s)$, where the $k_j$ run through the elements of 
$\mathcal{S} = \mathcal{S}_n$ in increasing order. Thus $\widehat{k}_i = k_{s+1-i}.$
Then  $\mathcal{U}={\mathcal{U}_n} $ is the  $s \times s$
integer matrix with  entries  $\mathcal{U}_{i,j} = \lfloor \frac{n}{k_i k_j} \rfloor$.
That is,
\[
{\mathcal{U}_n} = 
 \begin{pmatrix}
 \lfloor  n/(k_1k_1) \rfloor& \lfloor n/(k_1k_2)\rfloor & \cdots & \lfloor n/(k_1k_{s-1}) \rfloor & 
 \lfloor n/(k_1k_{s}) \rfloor \\
 \lfloor  n/(k_2k_1) \rfloor &  \lfloor n/(k_2k_2) \rfloor & \cdots & \lfloor  n/(k_2k_{s-1}) \rfloor& 0 \\
  \vdots  & \vdots  & \ddots & \vdots &  \vdots  \\
 \lfloor   n/(k_{s-1}k_1) \rfloor&  \lfloor n/(k_{s-1}k_2) \rfloor& \cdots & 0 & 0 \\
  \lfloor  n/(k_sk_1)\rfloor& 0 & \cdots & 0 & 0
 \end{pmatrix}
\]
\end{lem}

The fact that the matrix in Lemma \ref{le22b}  is symmetric
easily  follows from
the following  set of identities  for the floor function
\cite[Lemma 6]{C10E}.

\begin{prop}\label{pr22} {\em (Cardinal)}
For all  positive  integers $n, i, j$, there holds
$$
\lfloor \frac{1}{i} \lfloor \frac{n}{j} \rfloor \rfloor =
\lfloor \frac{1}{j} \lfloor \frac{n}{i} \rfloor \rfloor =
 \lfloor \frac{n}{ij} \rfloor.
$$
\end{prop}

These identities may be checked by a calculation. In fact
 that dilated floor function  $\lfloor \alpha x \rfloor$
 and $\rfloor \beta x\rfloor$ never commute for real $(\alpha, \beta)$
except for  the discrete family $(\alpha,\beta) = (\frac{1}{i}, \frac{1}{j})$ with $i, j$  positive integers
as given in Proposition \ref{pr22},
aside from three continuous families: $\alpha=\beta$, $\alpha=0$
and $\beta=0$,  see \cite{LMR15}. 

Note that  all entries on the anti-diagonal $i+j=s+1$ of $\mathcal{U}_n$ are $1$, since
$$
\mathcal{U}_{i, s+1- i} = \lfloor \frac{n}{k_i k_{s-i+1}} \rfloor = 1
$$

\begin{exa}\label{ex33a}
{\rm
For $n=16$ the  matrix $\mathcal{U}_n$   is 
(omitted entries are $0$)
 
\[ \mathcal{U}_{16}=
  \begin{tabular}{ |c| c c c c c c r | }
  \hline
    ${\mathcal{U}}_{16}$ & 1 & 2 & 3 & 4 & 5 & 8 & 16 \\
    \hline
    1  & 16   & 8   & 5  & 4   & 3  & 2 & 1 \\
    2  & 8    & 4   & 2   & 2   & 1   & 1 &   \\
    3  & 5  & 2 &  1  & 1 & 1 &   &   \\
    4  & 4    & 2  & 1  & 1   &       &   &   \\
    5  & 3 & 1 & 1 &     &       &   &   \\
    8  & 2    & 1   &       &     &       &   &   \\
    16 & 1    &     &       &     &       &   &   \\
    \hline
  \end{tabular}
\]
}
\end{exa}

\quad The matrix ${\mathcal{U}_n}$
is invertible, and since $\det( {\mathcal{U}_n}) = \pm 1$ it is
also an integer matrix.

\begin{exa} \label{ex23}
{\rm
 For the case $n=16$ the matrix ${\mathcal{U}_n}^{-1}$ is
(omitted  entries are $0$)
\[ (\mathcal{U}_{16})^{-1} \, \,=
  \begin{tabular}{ | c  | c c c c c c r | }
  \hline
    $({\mathcal{U}}_{16})^{-1}$ & 1 & 2 & 3 & 4 & 5 & 8 & 16 \\
    \hline
    1  &  &    &   &    &   & & 1 \\
    2  &      &   &    &   &    & 1 & -2  \\
    3  & &  &    &  & 1 &  -1 & -1  \\
    4  &    &  &   & 1   &   -1    & -1  &  1 \\
    5  &  &  & 1 &   -1  &  0     &  0 &   -1\\
    8  &     & 1   &    -1   &  -1   &     0  & 0  & 1  \\
    16 & 1    &  -2   &    -1   &  1   &   -1    &  1 &  2 \\
    \hline
  \end{tabular}
\]
}
\end{exa}


%
%

\section{Cardinal's Matrix $\mathcal{M}_n$ and the Mertens Function}\label{sec4}

Cardinal  also introduces an $s \times s$ symmetric matrix  $\mathcal{M}_n$, that corresponds
to the M\"{o}bius vector $\bmu$ similarly.  

\subsection{Cardinal's Matrix $\mathcal{M}_n$}\label{sec41}
\begin{defi} \label{de34a}
The $s \times s$ symmetric matrix  $\mathcal{M}_n$ is defined by
$$
\mathcal{M} = {\mathcal{M}}_n := \sT_s\, \mathcal{U}_n^{-1}\sT_n,
$$
in which $\sT_n$ is an $s(n) \times s(n)$  matrix having  
$ \sT_{ij} = 1$ if $ i+j \le s(n)+1$, and $0$
otherwise (i.e.  if $s(n) +2 \le i+j \le 2s(n)$). 
\end{defi}

Note that by definition of $\mathcal{U}_n$, one also has 
$$
\mathcal{M}_n= \sT_s (Z_n)^{-1}= \sT_s \rho (\bmu).
$$

\begin{exa}\label{ex42a}
{\rm
For  $n=16$ the matrix ${\mathcal{M}_n}$ is:
(omitted entries are  $0$)

\[ \mathcal{M}_{16} = 
  \begin{tabular}{ |c | c c c c c c r | }
  \hline
    ${\mathcal{M}}_{16}$ & 1 & 2 & 3 & 4 & 5 & 8 & 16 \\
    \hline
    1  & -1   & -2   & -2  & -1  & -1  & 0 & 1 \\
    2  & -2    & -1   & 0  & 0   & 1   & 1 &   \\
    3  & -2 & 0 &  1  & 1 & 1 &   &   \\
    4  & -1    & 0  & 1  & 1   &       &   &   \\
    5  & -1 & 1 & 1 &     &       &   &   \\
    8  & 0   & 1   &       &     &       &   &   \\
    16 & 1    &     &       &     &       &   &   \\
    \hline
  \end{tabular}
\]
}
\end{exa}

A main result of Cardianal is that the entries of this matrix 
are expressible using the {\em Mertens function}
$$
M(x) := \sum_{1 \le j \le \lfloor x \rfloor} \mu(j).
$$
Here we set $M(x)=0$ for $0 \le x <1.$

\begin{thm} {\em (Cardinal)}\label{th41}
The entries of $\mathcal{M}_n$ are exactly the Mertens function
evaluation of the entries of $\mathcal{U}_n$, i.e. 
$$
(\mathcal{M}_n)_{ij} = M( \mathcal{U}_{ij}), ~~~\mbox{for}~~ i, j \in \sS_n,
$$
where $M(x)$ is the Mertens function. That is
\[
{\mathcal{M}_n} = 
 \begin{pmatrix}
 M( \lfloor  n/(k_1k_1) \rfloor )&M( \lfloor n/(k_1k_2)\rfloor )& \cdots &M( \lfloor n/(k_1k_{s-1}) \rfloor )& 
M( \lfloor n/(k_1k_{s}) \rfloor )\\
 M(\lfloor  n/(k_2k_1) \rfloor )& M( \lfloor n/(k_2k_2) \rfloor )& \cdots &M( \lfloor  n/(k_2k_{s-1}) \rfloor )& 0 \\
  \vdots  & \vdots  & \ddots & \vdots &  \vdots  \\
 M(\lfloor   n/(k_{s-1}k_1) \rfloor )& M( \lfloor n/(k_{s-1}k_2) \rfloor )& \cdots & 0 & 0 \\
 M( \lfloor  n/(k_sk_1)\rfloor )& 0 & \cdots & 0 & 0
 \end{pmatrix}
\]
\end{thm}

Here again one has that all the anti -diagonal entries 
$i+j=s+1$ are given by
$$
\mathcal{M}_{i, s+1- i} =  M( \lfloor \frac{n}{k_i k_{s-i+1}}\rfloor ) = M(1)=   1.
$$

\subsection{Riemann Hypothesis equivalence in terms of $\mathcal{M}_n$}\label{sec42}

The Riemann hypothesis is nicely encoded in terms of   a suitable bound on the norms 
of these matrices
$\mathcal{M}_n$.

\begin{thm} \label{th43}
The following properties are equivalent.

(i) For each $\epsilon >0$ there holds the estimate 
$$
||\mathcal{M}_{n} ||_F = O_{\epsilon} (n^{\frac{1}{2} + \epsilon}).
$$
where $||\mathcal{M}||_F^2 = \sum _{i=1}^s \sum_{j=1}^s \mathcal{M}_{ij}^2$ is the
Frobenius matrix norm.\\

(ii) The Riemann hypothesis holds.
\end{thm}

\begin{proof}
$(i) \Rightarrow (ii). $   The hypothesis (i) yields the upper bound
$$
M(n)^2 = |\mathcal{M}_{11}|^2\le ||\mathcal{M}_n||_F^2  =  O_{\epsilon}\left( n^{1+2\epsilon} \right).
$$
This upper bound on the the Mertens function  is well-known to imply the Riemann hypothesis
\cite[Theorem 14.25(C)]{TH86}.\\

$(ii) \Rightarrow (i).$ The Mertens function RH bound implies that
\begin{eqnarray*}
 ||\mathcal{M}_n||_F^2 & = & \sum _{i=1}^s \sum_{j=1}^s \mathcal{M}_{ij}^2\\
&  \le & \sum_{i=1}^s \sum_{j=1}^s M( \frac{n}{k_ik_j})^2\\
& = & O_{\epsilon} 
\left( \sum _{i=1}^s \sum_{j=1}^s (\frac{n}{k_ik_j})^{1+ \epsilon}\right) \\
& = & O_{\epsilon}
\left(  n^{1+\epsilon}( \sum _{i=1}^{\infty} (\frac{1}{k_i})^{1+\epsilon})
(\sum_{j=1}^{\infty}(\frac{1}{k_j})^{1+ \epsilon})\right) \\
& =  &  O_{\epsilon} \left( n^{1+\epsilon}  \right).
\end{eqnarray*}
Here we used the fact that the Mertens function is constant on unit intervals to
remove the greatest integer function. The important thing is that the implied
constant in the $O$-symbol does not depend on $s=s(n)$.
\end{proof} 

Cardinal observes empirically that the function $||{\mathcal{M}}_n||$ 
(using the $l_2$ operator norm) seems to be a much
smoother function than the Mertens function $M(n)$, i.e. it empirically has smaller
fluctuations.


\begin{rem} \label{rek44a} 
The Frobenius norm $||\mathcal{M}_n||_F$
must be at least as large as its $(1,1)$-th entry.
Hence  it must see fluctuations
as large as those of $M(n)$ in the upwards direction.

The true order of growth of the Mertens function, assuming the truth of RH, is not known
at present.
 Soundararajan \cite{S09} 
 shows that the RH implies the upper bound
$$
|M(n)| = O\left( \sqrt{n} \exp( (\log x)^{1/2} (\log \log x)^{14} ) \right) 
$$
His paper suggests that assuming
conjectured bounds on the maximal size of L-functions
(made by Farmer, Gonek and Hughes)  one might be able to  derive
the stronger bound
$$
|M(n) |  =_{?} O\left( \sqrt{n} \exp\left( C(\log \log x)^{3} \right) \right).
$$
Heuristics suggest that the upper bound is much smaller.
Nathan Ng [Ng04, Theorem 1] shows, assuming unproved hypotheses, that 
$$
|M(n)| \le \sqrt{n} (\log n)^{\frac{3}{2}}.
$$
He  notes that S. Gonek has conjectured
that  the best possible maximal order of magnitude of $M(n)$ 
will be of shape
$$ |M(x)| \le B \sqrt{n} (\log\log\log n)^{5/4}, $$
where $B>0$ is a constant. 
 In the other direction it is  known
that the  GUE hypothesis implies that
$$
\limsup_{n \to \infty}  \frac{|M(n)|}{\sqrt{n}} = + \infty. 
$$
Kaczorowski \cite{Kac07} has proved a complementary result showing that
for a certain analogue  $M^{*}(n)$ of $M(n)$ the fluctuations of $M^{*}(n)$ are at
least of size $\sqrt{n} \log\log\log n$.

\end{rem}

%
%

\section{Deformed Version of  Cardinal matrix $\mathcal{U}_n$}\label{sec5}

We define and study certain deformed matrices $\widetilde{\mathcal{U}}_n$, of similar design to $\mathcal{U}_n$.
All entries of $\widetilde{\mathcal{U}}_n$ are greater than or equal to the corresponding entry of $\mathcal{U}_n$,
and any entry increase is less than one. Thus $\mathcal{U}_n$ is recoverable from  $\widetilde{\mathcal{U}}_n$
by applying the floor function to each of its entries.
These are modification of the deformation proposed in \cite{C10F}.

\subsection{Deformed Matrix $\widetilde{\mathcal{U}}_n$}\label{sec51}

We now define and study certain deformed matrices $\widetilde{\mathcal{U}}_n$, of similar design to $\mathcal{U}_n$.
All entries of $\widetilde{\mathcal{U}}_n$ are greater than or equal to the corresponding entry of $\mathcal{U}_n$,
and any increase is less than one. Thus $\mathcal{U}_n$ is recoverable from  $\widetilde{\mathcal{U}}_n$
by applying the floor function to each of its entries.

\begin{defi} \label{def51}
For $n \in \mathbb{N}$, let $\sT = \sT_s$ be the $s \times s$ matrix with $s= s (n)= \#(\mathcal{S}_n)$, 
which has  ones on and above the antidiagonal, and zeros elsewhere. 
That is, $\sT_s = (t_{ij})$, where $t_{ij} = 1$ if $i +j\leq s+1$, and 0 else. 
\[
\sT = 
 \begin{pmatrix}
  1 & 1 & \cdots & 1 & 1 \\
  1 & 1 & \cdots & 1 & 0 \\
  \vdots  & \vdots  & \ddots & \vdots &  \vdots  \\
  1 & 1 & \cdots & 0 & 0 \\
  1 & 0 & \cdots & 0 & 0
 \end{pmatrix}
\]

Its inverse 
$\sT^{-1}$ is the matrix with $1$'s on the antidiagonal and $(-1)$'s just below the antidiagonal. That is:

\[
 \sT^{-1} = 
 \begin{pmatrix}
 0  &  0& \cdots &0&  0 & 1 \\
  0 &0 &  \cdots &0 & 1 & -1 \\
  0&  0& \cdots & 1 & -1& 0\\
  \vdots  & \vdots  & \ddots & \vdots &  \vdots  \\
  0 & 1 & \cdots & 0 & 0 &0\\
  1 & -1 & \cdots & 0 & 0&0
 \end{pmatrix}.
\]
\end{defi}

\begin{defi}\label{def52}
For $n \in \mathbb{N}$, define the row vector 
$$\bd=( d_1, d_2, ..., d_s) := \sqrt{n}/k_1, \sqrt{n}/k_2, \cdots, \sqrt{n}/k_s),$$
where $k_j$ runs  through $
\mathcal{S} = \mathcal{S}_n$ in increasing order.
Let $D$ be the diagonal matrix $D = \rm{diag}(d_1,...,d_s)$, and then set
$$
\widetilde{\mathcal{U}} = \widetilde{\mathcal{U}}_n := DTD.
$$
Then set
$$\widetilde{\mathcal{M}} = \widetilde{\mathcal{M}}_n := T\widetilde{\mathcal{U}}^{-1}T.$$
\end{defi}

We note an equivalent definition of $\widetilde{\mathcal{U}}_n.$

\begin{prop}\label{prop53}
The matrix $\widetilde{\mathcal{U}}_n$ is equal to
\[
 \widetilde{\mathcal{U}}_n = 
 \begin{pmatrix}
  n/(k_1k_1) & n/(k_1k_2) & \cdots & n/(k_1k_{s-1}) & n/(k_1k_{s}) \\
  n/(k_2k_1) & n/(k_2k_2) & \cdots & n/(k_2k_{s-1}) & 0 \\
  \vdots  & \vdots  & \ddots & \vdots &  \vdots  \\
  n/(k_{s-1}k_1) & n/(k_{s-1}k_2) & \cdots & 0 & 0 \\
  n/(k_sk_1) & 0 & \cdots & 0 & 0
 \end{pmatrix}
\]
where $k_i$ runs through $\mathcal{S}$ in increasing order. That is, $\widetilde{\mathcal{U}} = (\widetilde{u}_{ij})$, where 

\[\widetilde{u}_{ij} = \left\{\begin{array}{lr}
       n/(k_ik_j) & : i +j \le s+1 \\
       0 & i +j \ge s+2
     \end{array}
   \right.
\]
\end{prop}
Note the  simple relationship:
 $\mathcal{U}_n$
is obtained from $\widetilde{\mathcal{U}}_n$ by taking the greatest integer part of each entry of $\widetilde{\mathcal{U}}_n$. \\

\begin{proof}
As defined, $\widetilde{\mathcal{U}} = D \sT D$, and we have
\begin{eqnarray*}
 D\sT D  &= &
 \begin{pmatrix}
  d_1 & 0 & \cdots & 0 & 0 \\
  0 & d_2 & \cdots & 0) & 0 \\
  \vdots  & \vdots  & \ddots & \vdots &  \vdots  \\
  0 & 0 & \cdots & d_{s-1} & 0 \\
  0 & 0 & \cdots & 0 & d_s
 \end{pmatrix}
  \begin{pmatrix}
  1 & 1 & \cdots & 1 & 1 \\
  1 & 1 & \cdots & 1 & 0 \\
  \vdots  & \vdots  & \ddots & \vdots &  \vdots  \\
  1 & 1 & \cdots & 0 & 0 \\
  1 & 0 & \cdots & 0 & 0
 \end{pmatrix}
  \begin{pmatrix}
  d_1 & 0 & \cdots & 0 & 0 \\
  0 & d_2 & \cdots & 0) & 0 \\
  \vdots  & \vdots  & \ddots & \vdots &  \vdots  \\
  0 & 0 & \cdots & d_{s-1} & 0 \\
  0 & 0 & \cdots & 0 & d_s
 \end{pmatrix}\\
&=  & \begin{pmatrix}
  d_1 & d_1 & \cdots & d_1 & d_1 \\
  d_2 & d_2 & \cdots & d_2 & 0 \\
  \vdots  & \vdots  & \ddots & \vdots &  \vdots  \\
  d_{s-1} & d_{s-1} & \cdots & 0 & 0 \\
  d_s & 0 & \cdots & 0 & 0
 \end{pmatrix}
  \begin{pmatrix}
  d_1 & 0 & \cdots & 0 & 0 \\
  0 & d_2 & \cdots & 0 & 0 \\
  \vdots  & \vdots  & \ddots & \vdots &  \vdots  \\
  0 & 0 & \cdots & d_{s-1} & 0 \\
  0 & 0 & \cdots & 0 & d_s
 \end{pmatrix}\\
& =  & \begin{pmatrix}
  d_1d_1 & d_1d_2 & \cdots & d_1d_{s-1} & d_1d_s \\
  d_2d_1 & d_2d_2 & \cdots & d_2d_{s-1} & 0 \\
  \vdots  & \vdots  & \ddots & \vdots &  \vdots  \\
  d_{s-1}d_1 & d_{s-1}d_1 & \cdots & 0 & 0 \\
  d_sd_1 & 0 & \cdots & 0 & 0
 \end{pmatrix}.
\end{eqnarray*}
The proposition follows from noting that the entries of this final matrix are equal to the $\widetilde{u}_{ij}$ defined in the proposition statement.
\end{proof}

\begin{exa}\label{ex54a}
{\rm
For the case $n=16$ the  matrix $\widetilde{\mathcal{U}}_{n}$  is  
(omitted entries are  $0$ )
\[ \widetilde{\mathcal{U}}_{16} =
  \begin{tabular}{ |c | c c c c c c r | }
  \hline
    $\widetilde{\mathcal{U}}_{16}$ & 1 & 2 & 3 & 4 & 5 & 8 & 16 \\
    \hline
    1  & 16   & 8   & 16/3  & 4   & 16/5  & 2 & 1 \\
    2  & 8    & 4   & 8/3   & 2   & 8/5   & 1 &   \\
    3  & 16/3 &  8/3& 16/9  & 4/3 & 16/15 &   &   \\
    4  & 4    & 2   & 4/3  & 1   &       &   &   \\
    5  & 16/5 & 8/5 & 16/15 &     &       &   &   \\
    8  & 2    & 1   &       &     &       &   &   \\
    16 & 1    &     &       &     &       &   &   \\
    \hline
  \end{tabular}
\]
}
\end{exa} 
The matrix $\widetilde{\mathcal{U}}_n$ can be viewed as a one-sided perturbation of
$\mathcal{U}_n$ in which each entry is larger by some amount between $0$ and $1$.\\

We next define a  matrix 
$\widetilde{\mathcal{U}}^{+}= \widehat{\mathcal{U}}_n^{+} := \bd^T \bd$ (outer product),
which has entries
\[
\widehat{\mathcal{U}}_{ij}^{+} =  n/(k_ik_j),   \quad 1\le i, j \le s.
\]
In the French preprint  \cite{C10F} this was Cardinal's formal definition of a 
matrix which he denoted $\tilde{\mathcal{U}}$. This
matrix has rank one, in consequence
 Cardinal's  formal definition of $\tilde{\mathcal{M}}$, which uses  $\tilde{\mathcal{U}}^{-1}$, becomes undefined.

\begin{exa}\label{ex55a}
{\rm
For the case $n=16$ the  outer product matrix $\widetilde{\mathcal{U}}_{16}^{+}= \bd^T \bd$  is:
\[ \widetilde{\mathcal{U}}_{16}^{+} =
  \begin{tabular}{| c | c c c c c c r | }
  \hline
    $\widetilde{\mathcal{U}}_{16}^{+}$ & 1 & 2 & 3 & 4 & 5 & 8 & 16 \\
    \hline
    1  & 16   & 8   & 16/3  & 4   & 16/5  & 2 & 1 \\
    2  & 8    & 4   & 8/3   & 2   & 8/5   & 1 &  1/2 \\
    3  & 16/3 &  8/3& 16/9  & 4/3 & 16/15 &  2/3 & 1/3  \\
    4  & 4    & 2   & 4/3  & 1   &  4/5     &  1/2 & 1/4  \\
    5  & 16/5 & 8/5 & 16/15 &  4/5   &  16/25     & 2/5  &  1/5 \\
    8  & 2    & 1   &    2/3   &  1/2   &  2/5     &  1/4 &  1/8 \\
    16 & 1    &  1/2   &    1/3   &   1/4  &   1/5    &  1/8 & 1/16\\
    \hline
  \end{tabular}
\]
This matrix has rank one, and agrees with $\widetilde{\mathcal{U}}_{16}$ on and 
above the antidiagonal. 
The difference matrix  is:
\[ \widetilde{\mathcal{U}}_{16}^{+}-   \widetilde{\mathcal{U}}_{16} =
 \begin{tabular}{ |c | c c c c c c r | }
 \hline
    $\widetilde{\mathcal{U}}_{16}^{+}-   \widetilde{\mathcal{U}}_{16}$ & 1 & 2 & 3 & 4 & 5 & 8 & 16 \\
    \hline
    1  &    &    &  &    &  &  & 0 \\
    2  &     &   &    &    &   & 0 &  1/2 \\
    3  &  &  &   &  & 0&  2/3 & 1/3  \\
    4  &    &   &   & 0   &  4/5     &  1/2 & 1/4  \\
    5  &  & & 0&  4/5   &  16/25     & 2/5  &  1/5 \\
    8  &     & 0   &    2/3   &  1/2   &  2/5     &  1/4 &  1/8 \\
    16 & 0    &  1/2   &    1/3   &   1/4  &   1/5    &  1/8 & 1/16\\
    \hline
  \end{tabular}
\]
}
\end{exa}

The matrix $\widetilde{\mathcal{U}}_n$ can be viewed as a one-sided perturbation of
$\mathcal{U}_n$ in which each entry is larger by some amount between $0$ and $1$.

\begin{rem}\label{rmk56}
(1) On can recover the Cardinal matrix $\mathcal{U}_n$ from either of 
the matrices $\widetilde{\mathcal{U}_n^{+}}$ and $\widetilde{\mathcal{U}}_n$
 by applying the floor function to each entry.
Here  $\widetilde{\mathcal{U}}_n$ is also obtained by applying the floor
function to some entries of $\widetilde{\mathcal{U}_n^{+}}$, namely those 
 entries
of $\widetilde{\mathcal{U}_n^{+}}$that are  strictly
below the antidiagonal (the lower right corner), and 
whose effect is to  zero out all of these entries. 
 As a consequence $\widetilde{\mathcal{U}}$  is an invertible matrix,
 having determinant the product of the entries on the anti-diagonal, times $(-1)^{\frac{n(n-1)}{2}}$.
 Furthermore all entries on the antidiagonal are $\ge 1$, so that $|\det (\widetilde{\mathcal{U}})| \ge 1$.

(2) The structure of ${\mathcal{U}_n}^{-1}$ is similar in shape to that of $T^{-1}$. Namely
it is supported on the antidiagonal and the parallel antidiagonal under it, see the next example.
\end{rem} 

\begin{exa}\label{ex57}
{\rm 
For the case $n=16$ the  matrix $\widetilde{\mathcal{U}}_{n}^{-1}$  is  
(omitted entries are $0$)
\[ \widetilde{\mathcal{U}}_{16} =
  \begin{tabular}{ | c | c c c c c c r | }
  \hline
    $\widetilde{\mathcal{U}}_{16}$ & 1 & 2 & 3 & 4 & 5 & 8 & 16 \\
    \hline
    1  &    &    &   &    &   &  & 1 \\
    2  &    &   &   &   &    & 1 &  -1/2 \\
    3  &  &  &   &  & 15/16 & -3/2  &   \\
    4  &     &   &   & 1   &    -5/4   &   &   \\
    5  &  &  & 15/16 & -5/4    &       &   &   \\
    8  &     & 1   &   -3/2    &     &       &   &   \\
    16 & 1    & -1/2    &       &     &       &   &   \\
    \hline
  \end{tabular}
\]
}
\end{exa}

%
%

\section{ Deformed  Matrix $\widetilde{\mathcal{M}}_n$}\label{sec6}
 In  this section we show  the  proof idea of Cardinal on the deformed matrix 
given in \cite{C10F} can be adjusted to 
  work using the modified definition of  $\tilde{\mathcal{U}}_n$ given above.
 The argument given follows \cite{C10F} closely, but uses the Frobenius norm
 instead of the $\ell_2$ operator norm. \\
   
%
%

 \subsection{Deformed matrix $\widetilde{\mathcal{M}}_n$}

We define the  {\em deformed Cardinal Mertens  matrix} $\widetilde{\mathcal{M}}_n$
by following the same recipe for $\mathcal{M}$, namely
$$\widetilde{\mathcal{M}}_n := T (\widetilde{\mathcal{U}}_n)^{-1}.$$

\begin{prop}\label{prop58} 
For the nonnegative symmetric $s \times s$ real matrices $\sT_n, \mathcal{U}_n, \widetilde{\mathcal{U}}_n$
the following hold.\\
\ben
\item $\sT \leq_{P} \mathcal{U} \leq_{P}  \widetilde{\mathcal{U}}$, where $\leq_{P}$  represents the inequality of corresponding entries. \\

\item For the values of $\mathcal{U}$ in the upper-left corner of the matrix, we have 
$\mathcal{U} \simeq \widetilde{\mathcal{U}}$,
in the sense that for $i+j \le s+1$, 
$$0 \le \widetilde{\mathcal{U}}_{ij} - \mathcal{U}_{ij} \le 1 = T_{ij} \le \mathcal{U}_{ij}.$$ 

\item For the values of $\mathcal{U}$ on the antidiagonal $i+j= s+1$, we have $\mathcal{U}_{ij} = \sT_{ij}=1$. \\
\een
\end{prop}

\begin{proof}
The properties (1)-(3) follow by inspection.
\end{proof}

%
%

\subsection{Bound for norm of deformed matrix $||\widetilde{\mathcal{M}}_n||_F$}\label{sec62}

Our  main observation  is that the norm of $||\widetilde{\mathcal{M}}_n||_F$ can be
unconditionally  estimated, and it satisfies the Riemann hypothesis bound.

To  estimate the growth of $||\widetilde{\mathcal{M}}_n||_F$, we first need a bound for $||\sT_n||_F.$\\

\begin{prop} \label{prop59}  The matrix $T= T_{s(n)}$ satisfies the Frobenius norm bound
\[ ||\sT_n||_F = O(\sqrt{n}), \]
\end{prop}

\begin{proof}
For a $s \times s$ matrix $A$, let $\max |A| = \max |a_{ij}|$. We use the general bound 
$$||A||_F^2  = \sum_{i=1}^s\sum_{j=1}^s |A_{ij}|^2 \le s^2 (\max |A|)^2, $$
which gives
$$||A||_F \le s \max|A|.$$
 Here, applied to  $A = \sT_n$, we have $\max |A| = 1$, and $s = \#S \sim 2\sqrt{n}$.
\end{proof}

\begin{rem}\label{rmk47}
It is clear that $||\sT_n||_F \ge \sqrt{n},$ since $\sT_n$ has $\sim \frac{1}{2} (2\sqrt{n})^2$ entries equal to $1$.
In terms of the $l_2$-norm, it is also
easy to see that $\liminf ||\sT_n||/\sqrt{n} \geq 1$. Indeed, let us consider the column vector $w$ of size $\#S$ made up of 1s, and denote its transpose by $w'$. Because of the fact that $\#S \simeq 2\sqrt{n}$
we have $||w||^2 \simeq 2\sqrt{n}$ and also $w' \sT w \simeq 2n$. The spectral radius of 
the symmetric matrix $\sT_n$, which is also the $l_2$-norm of $\sT_n$, is thus greater than $\frac{w'Tw}{||w||^2} \simeq \sqrt{n}$.
\end{rem}

Now we can bound the size of the perturbed inverse matrix $\widetilde{\mathcal{M}}_n$ elementwise, as follows.

%
\begin{lem}\label{lem511} The deformed matrix $\widetilde{\mathcal{M}}_n$ has elements bounded by
\[ \max_{i,j}  |(\widetilde{\mathcal{M}}_n)_{i,j}| = O(\log n). \]
\end{lem}
\begin{proof}
As $\widetilde{\mathcal{U}}_n$ is defined to be $D\sT D$, we have $(\widetilde{\mathcal{U}}_n)^{-1} = D^{-1} \sT^{-1}D^{-1}$. We now calculate $\sT^{-1}$, $D^{-1}$, and $(\widetilde{\mathcal{U}}_n)^{-1}$.

\ben
\item $\sT^{-1}$ is the matrix with $1$'s on the antidiagonal and $(-1)$'s just below the antidiagonal. That is:

\[
 \sT^{-1} = 
 \begin{pmatrix}
   &  &&  &&  & 1 \\
   & & & & & 1 & -1 \\
  &  & & & 1 & -1& \\
    &  & &\ddots &  &    \\
    &  &1& &  &  &\\
   
   & 1 & -1& &  &  &\\
  1 & -1 & & &  & &
 \end{pmatrix}.
\]

\item $D^{-1} = \rm{diag}(1/\sqrt{n},...,k/\sqrt{n},...,n/\sqrt{n})$, where $k$ runs through $\mathcal{S}$. \\

\item $(\widetilde{\mathcal{U}}_n)^{-1}$ has the same shape as $\sT^{-1}$ in that the only nonzero entries are on and just below the antidiagonal. 

Traversing the antidiagonal of $(\widetilde{\mathcal{U}}_n)^{-1}$ from the bottom left to the top right, we get the terms $\frac{k\overline{k}}{n}$, where $k$ runs through $\mathcal{S}$. As $\overline{k} = [n/k]$, it follows that each one of these terms lies between 0 and 1. 

Traversing just under the antidiagonal, we have the terms $-\frac{\overline{k}k^+}{n}$, where $k$ runs through 
$\mathcal{S} \setminus \{n\}$ and $k^+$ designates the successor of $k$ in $\mathcal{S}$. 
As we can bound $\frac{\overline{k}k^+}{n}$ by $\frac{(n/k)k^+}{n} = \frac{k^+}{k} \leq C$ 
for some constant $C$ (as shown in point 3 of Proposition 2.4 of \cite{C10F}, 
we can take $C = 4 + 2\sqrt{2}$). Thus each of these terms thus lies between $-C$ and 0.
\een

Now, we turn to $\widetilde{\mathcal{M}}_n = \sT(\mathcal{U}_n)^{-1}\sT$. 
Upon multiplying a general $s \times s$ matrix $A = (a_{ij})$ on both sides by $\sT$, 
we see that the resulting matrix takes the form $\sT A \sT = C = (c_{ij})$, where 
\[ c_{ij} = \sum_{k = 1}^{s+1-i}\sum_{\ell = 1}^{s+1-j}a_{k\ell}. \]

Thus, we see that obtaining the $(i,j)$-th  term of $\widetilde{\mathcal{M}}_n = \sT (\widetilde{\mathcal{U}}_n)^{-1} \sT$
 consists of summing all the terms within the $(s+1-i) \times (s+1-j)$ submatrix starting located 
 at the top left of $(\widetilde{\mathcal{U}}_n)^{-1}$.
  Considering the form of $(\widetilde{\mathcal{U}}_n)^{-1}$ -- the only nonzero entries 
  are on and just below the antidiagonal -- and the fact that the matrix is symmetric, 
  we conclude that each coefficient of $\widetilde{\mathcal{M}}_n$ is the sum of at most:

\bi

\item two sums of terms along the antidiagonal, each of the form (the two sums are the
 same as the matrix is symmetric) 
 $\frac{1}{n}\sum_{k \in \mathcal{S}, i \leq k \leq j}\overline{k}(k-k^+)$, 
 where $i$ and $j$ are fixed in $\mathcal{S}$, and $j < \sqrt{n}$ \\

\item and a sum of at most three terms at the center of the matrix, 
each of which is between $-C$ and $1$.

\ei

Now, since we have $k < \sqrt{n}$ in the above sums, we have that $k^+ = k+1$
and so each of the two sums mentioned above simplifies to 
$-\frac{1}{n}\left( \sum_{k \in \mathcal{S}, i \leq k \leq j}\overline{k}\right)$. 
Now, we have the bound

\[ \sum_{k \in \mathcal{S}, i \leq k \leq j}\overline{k} 
\leq \sum_{1 \leq k \leq \sqrt{n}}\overline{k} =
 \sum_{1 \leq k \leq \sqrt{n}}[n/k] \leq \sum_{1 \leq k \leq \sqrt{n}}n/k \sim n \log \sqrt{n}, \]

and so we obtain

\[ \max |\widetilde{\mathcal{M}}_n| = O(\log n). \]

\end{proof}

Lemma \ref{lem511} immediately yields a bound for the 
Frobenius matrix norm $ ||\widetilde{\mathcal{M}}_n||_F$.

\begin{cor} \label{cor512} The Frobenius norm of $\widetilde{\mathcal{M}}_n$ satisfies
\[ ||\widetilde{\mathcal{M}}_n||_F = O(\sqrt{n}\log n). \]
\end{cor}
\begin{proof}
Use Lemma \ref{lem511} along with the inequality used in 
Proposition \ref{prop58}, which is $||A||_F \leq s \max |A|$.
\end{proof}


\begin{rem} \label{rek513a} The  Frobenius norm upper bound $O (\sqrt{n}\log n)$ for $\widetilde{\mathcal{M}}_n$
is better   than the known $\Omega$-bounds
on the fluctuations of the Mertens's function, assuming RH. See Remark \ref{rek44a}.
\end{rem}

\begin{rem}\label{rek514}
Cardinal \cite{C10E} reports 
that  numerical experiments with the $l_2$-operator norm suggest that  the ratio $||\widetilde{\mathcal{M}}_n||/(\sqrt{n}\log n)$
 seems to tend to 0. 
He also remarks  it is easy to prove that $\max |\widetilde{\mathcal{M}}_n|/\log n$
 has a strictly positive limit (a result more precise than Lemma \ref{lem511}). It appears that the inequality
 $||A|| \leq s \max |A|$ gives away too much to deduce such a result.
\end{rem}


\section{ Remarks on  Bounding  Cardinal's matrix  $||{\mathcal{M}}_n||_F$}\label{sec7}

We would like to obtain upper bounds for   $||{\mathcal{M}}_n||_F$, perhaps viewing ${\mathcal{M}}_n$
as a perturbation of $\widetilde{\mathcal{M}}_n$. 
One approach is  first study  its inverse $\, \mathcal{U}_n$ compared to
$\, \widetilde{\mathcal{U}}_n$.  The basic quantity  controlling the
size of $||{\mathcal{M}}_n||_F$ will be the size of the {\em smallest} eigenvalue of $\mathcal{U}_n$.
This is because $\mathcal{M}_n$ is  a fixed linear of the matrix $ T(\mathcal{U}_n)^{-1}$
and the largest norm eigenvalue of the symmetric matrix $(\mathcal{U}_n)^{-1}$ is 
the reciprocal of the smallest norm (nonzero) eigenvalue of $\mathcal{U}_n)$.
Note that $\det(\mathcal{U}_n) =1$ and $\det( \widetilde{\mathcal{U}}_n) \ge 1$.\\

\subsection{Positivity property}\label{sec71}

 Each  of
$\mathcal{U}_n$ and $\, \widetilde{\mathcal{U}}_n$ and, especially,
their difference matrix
$$
E_n := \widetilde{\mathcal{U}}_n -\mathcal{U}_n
$$
are {\em nonnegative} symmetric matrices. Perhaps the nonnegativity
constraint can be put to some use. Note that the
perturbation $E_n$ has $\det(E_n) =0$.\\

\begin{exa}\label{ex61a}
{\rm 
For the case $n=16$ the  matrix $E_n$  is  
(omitted entries are  $0$)
\[ E_{16} =
  \begin{tabular}{ |c| c c c c c c r | }
  \hline
    $E_{16}$ & 1 & 2 & 3 & 4 & 5 & 8 & 16 \\
        \hline
    1  & 0   & 0   & 1/3  & 0  & 1/5  & 0 & 0 \\
    2  & 0   & 0   & 2/3   & 0   & 3/5   & 0 &  \\
    3  & 1/3 &  2/3& 7/9  & 1/3 & 1/15 &  &  \\
    4  & 0    & 0   & 1/3  & 0   &      &    &   \\
    5  & 1/5 & 3/5 & 1/15 &     &       &   &   \\
    8  & 0    & 0   &    &     &       &   &   \\
    16 & 0    &     &     &    &      &  &   \\
    \hline

    \end{tabular}
\]
}
\end{exa}

Next we note that Cardinal's original deformed
matrix $\widetilde{\mathcal{U}}_n^{+}$ is given by the outer product
$$
\widetilde{\mathcal{U}}_n^{+} := [d_1, d_2, \cdots , d_s] ^T [d_1 d_2, ..., d_s] = [d_i d_j]_{1 \le i, j \le s}\,
$$
which shows it is  a rank one matrix, hence non-invertible.  It satisfies 
$$
\widetilde{\mathcal{U}}_n^{+} \ge_{N}  \widetilde{\mathcal{U}}_n \ge_{N} 
\mathcal{U}_n.
$$
where $A \ge_{N} B$ means $A- B$ is a nonnegative (symmetric) matrix.
It follows that the larger difference matrix
$$
\widetilde{E}_n := \widetilde{\mathcal{U}}_n^{+} -\mathcal{U}_n
$$
exhibits  a sufficiently large nonnegative perturbation to reach a very singular matrix
$\widetilde{\mathcal{U}}_n^{+}$ from $\mathcal{U}_n$.\\

\begin{exa}\label{ex62a}
{\rm
For the case $n=16$ the  larger difference matrix $\widetilde{E}_n$  is  
\[ \widetilde{E}_{16} =
  \begin{tabular}{ | c | c c c c c c r | }
  \hline
    $\widetilde{E}_{16}$ & 1 & 2 & 3 & 4 & 5 & 8 & 16 \\
    \hline
    1  & 0   & 0   & 1/3  & 0  & 1/5  & 0 & 0 \\
    2  & 0   & 0   & 2/3   & 0   & 3/5   & 0 & 1/2  \\
    3  & 1/3 &  2/3& 7/9  & 1/3 & 1/15 & 2/3  &  1/3 \\
    4  & 0    & 0   & 1/3  & 0   &    4/5   & 1/2   &  1/4 \\
    5  & 1/5 & 3/5 & 1/15 &  4/5   &   16/25    &  2/5 &  1/5 \\
    8  & 0    & 0   &   2/3    &  1/2   &  2/5     & 1/4  & 1/8  \\
    16 & 0    & 1/2    &   1/3    &  1/4   &   1/5    &  1/8 & 1/16  \\
    \hline
  \end{tabular}
\]
}
\end{exa}

This suggests that it will be hard to bound the smallest eigenvalue
of $\mathcal{U}_n$ by a general matrix inequality. \\

It may also be useful to  study  the difference matrix
$$
E_n^{+} :=\widetilde{\mathcal{U}}_n^{+}  -\widetilde{\mathcal{U}}_n.
$$
Here on has
$$
\tilde{E}_n = E_n^{+} + E_n.
$$

\begin{exa}\label{ex63}
{\rm
For the case $n=16$ the   difference matrix $E_n^{+}$  is  
(omitted entries are  $0$ )
\[ E_{16}^{+} =
  \begin{tabular}{ | c | c c c c c c r | }
  \hline
    $E_{16}^{+}$ & 1 & 2 & 3 & 4 & 5 & 8 & 16 \\
    \hline
    1  &    &    &   &   &   &  & 0 \\
    2  &    &    &   &   &    & 0 & 1/2  \\
    3  &  &  &   &  & 0 & 2/3  &  1/3 \\
    4  &    &    &   & 0   &    4/5   & 1/2   &  1/4 \\
    5  &  & & 0 &  4/5   &   16/25    &  2/5 &  1/5 \\
    8  &    & 0   &   2/3    &  1/2   &  2/5     & 1/4  & 1/8  \\
    16 & 0    & 1/2    &   1/3    &  1/4   &   1/5    &  1/8 & 1/16  \\
    \hline
  \end{tabular}
\]
}
\end{exa}

We can also define an upper triangular matrix $\widetilde{Z}_n$ by
$$
\widetilde{\mathcal{U}}_n = T_n \widetilde{Z}_n
$$
The matrix $\widetilde{Z}_n$ is lower triangular, and we can compare it
with $Z_n$. In comparing these two matrices, in the range $1 \le j \le i \le s$ we have
$$
-1 < (\widetilde{Z}_n - Z_n)_{i,j} <1,
$$
a result which follows from the structure of $T_n^{-1}$.

\begin{exa} \label{ex64}
{\rm 
For $n=16$ the  matrix $\widetilde{Z}_{n}$  is
(omitted entries are $0$)
\[ \widetilde{Z}_{16} =
  \begin{tabular}{ | c | c c c c c c r | }
  \hline
    $\widetilde{Z}_{16}$ & 1 & 2 & 3 & 4 & 5 & 8 & 16 \\
    \hline
    1  &  1 &   &  &    &  & &  \\
    2  & 1    &  1  &    &    &   &  &   \\
    3  &  6/5 & 3/5 &   16/15 & &  &   &   \\
    4  &   4/5 &  2/5   &   4/15 &   1   & &  &   \\
    5  &   4/3  &  2/3 &   4/9 &   1/3  &  16/15    &   &   \\
    8  &   8/3 &   4/3 &    8/9  &  2/3   &   8/15    &  1  &   \\
    16 &  8   &  4   &   8/3   &  2   &  8/5   & 1  &  1 \\
    \hline
  \end{tabular}
\]

}
\end{exa}

\begin{exa} \label{ex65}
{ \rm
For $n=16$ the  matrix $W_{n}:= \widetilde{Z}_{n}- Z_{n}$ f is
(omitted entries are $0$)
\[ W_{16}:= \widetilde{Z}_{16}- Z_{16}=
  \begin{tabular}{ | c | c c c c c c r | }
  \hline
    $W_{16}$ & 1 & 2 & 3 & 4 & 5 & 8 & 16 \\
    \hline
    1  &  0 &   &  &    &  & &  \\
    2  & 0    &  0  &    &    &   &  &   \\
    3  &  1/5 & 3/5 &  1/15 & &  &   &   \\
    4  &    -1/5 & -3/5   & 4/15 &   0   & &  &   \\
    5  &    1/3 &  2/3 &   4/9 &   1/3  &  1/15    &   &   \\
    8  &  -1/3 &  - 2/3 &   - 1/9  & - 1/3   &   8/15   &  0  &   \\
    16 &  0  &  0   &   - 1/3   &  0   &  -2/5  & 0  &  0 \\
    \hline
  \end{tabular}
\]
As remarked above, all  entries of $W_{n}$ lie strictly between $-1$ and $1$. The column sums of
$W_n$ are all nonnegative, since they can be shown to coincide with the first row of
$\widetilde{\mathcal{U}}_n - U_n$.
}
\end{exa}

\subsection{Continuous  deformations}\label{sec72}

To understand the structure of eigenvalues and eigenvectors of Cardinal's matrices
$\mathcal{U}_n$ one might try deforming the Cardinal matrix $\mathcal{U}_n$
in other fashions,  along a smooth path,
in which it remains skew upper triangular, and symmetric, 
with numbers close to $1$ on the diagonal.
These are homotopy paths and it might be interesting to use $\sT= \sT_{s(n)}$ 
as the base point. 

 One  can get to the matrix $\mathcal{U}_n$ starting from the matrix $\sT$
by making  non-negative increasing changes
in entries above the anti-diagonal, leaving zeros below the diagonal,
and only allowing such deformations  with all deformed matrices remaining
symmetric. This preserves the reality of the eigenvalues during the
deformation, which vary continuously. The eigenvalues of $T$ are all equal to $1$ and $-1$,
as equal as possible in number. They may mismatch in number by at most $1$, noting that
the trace of $T$ is $0$ or $1$ depending on whether $s$ is even or odd.

 Under a continuous  symmetric deformation, no eigenvalue can change sign,
because a symmetric matrix has real eigenvalues and eigenvalue $0$ cannot occur
because the anti-diagonal elements remain positive, so the sign of the determinant
stays fixed.

We  note that there are  about $s$ different 
values of $n$ giving Cardinal matrices $\mathcal{U}_n$ having the same size $s$,
because $s(n) $ jumps by $1$ only at values $n =k(k+1)$ or $k^2$. The Cardinal
algebras $\sA_n$ of these different $n$ are not the same in general. 
However the  number of such
matrices is much less than the total set of deformation parameters that maintain symmetry
of the matrix, which is about $\frac{1}{4} s^2$.
An interesting question is the structure of   deformations for which there is an underlying
rank $s$ commutative
algebra of lower triangular matrices (obtained by applying $\sT^{-1}$ that 
simultaneously is being deformed  along the
deformation path. Perhaps requiring such an  extra property would restrict the allowable
set of deformations to better match the number of sample matrices.

\section{ Concluding Remarks}\label{sec8}

The Riemann hypothesis can also be related to the
 the size of the smallest eigenvalue
(in absolute value) of $\mathcal{U}_n$. Its truth 
requires that this eigenvalue  not get too close to $0$. This property might be viewed as  a kind of level
repulsion phenomenon.
It might be useful to get more general information 
 about the eigenvalues of $\mathcal{U}_n$.
\begin{enumerate}
\item
How do the positive eigenvalues and (absolute values of )negative eigenvalues 
of $\mathcal{U}_n$ interact. Do they interlace?
\item
Does  interlacing of eigenvalues hold 
 when increasing
$n$ to $n+1$, or whether they change in a simple way.
Only a few entries in the matrix $\mathcal{U}_n$ change
in going from $n$ to $n+1$.
\item
 What happens to the eigenvalues at the
special values $n= k^2$ or $k(k+1)$  where the size of the Cardinal matrices increases by one.
One can break this jump in half by building additional  Cardinal matrices using
the ceiling function instead of the floor function. This adds another row
and column to the matrix and some interesting features emerge.
\end{enumerate}

The Cardinal algebra $\sA_n$  is a kind of ``finite-dimensional  quantization" 
of the algebra of Dirichlet series. it is  an interesting construction half way between
addition and multiplication. It has  added in a nice way  ``approximate divisors" which increase
the number of divisors of $n$ from $d(n)$ to  about $2 \sqrt{n}$. Recall that  
the number of divisors function  $d(n)=O(n^{\epsilon})$ for any $\epsilon >0$. The commutativity property of
the resulting algebra seems very important. The symmetry property of the matrices may
be a finite-dimensional vestige of    the symmetry
under $s \to 1-s$ given in the functional equation of the Riemann zeta function.
If that were so, then the matrix $T$ might be associated with the Euler factor
at the real place.

We now remark on various numerology connected with $n$ and $n+1$ together
that has appeared recently in several different number-theory  contexts.  
In a paper of the first author with Harsh Mehta \cite{LM15} we observed in a context of 
products of Farey fractions, but also potentially related to the Riemann hypothesis,  functions with jumps
that occur at a subset of these  values $k^2$ and $k(k+1)$,  which arise in part as an artifact of  
Dirichlet's hyperbola method.  In another direction, work on splitting of polynomials
of the first author with B. L. Weiss \cite{LagW15} led to the discovery by interpolation in a variable $z$ of
measures defined on each symmetric group $S_n$ 
 at $z=p$, a prime, to a signed measure  at  
on the symmetric group $S_n$  at the value $z=1$,
 which combines symmetric group representations from $S_n$ and $S_{n-1}$
in an interesting way, and has an internal  multiplicative structure respecting
integer multiplication on $n$.  
This is being explored further in  work in progress with Trevor Hyde.

Might there  exist a family of finite-dimensional quantum  integrable systems,
with parameter $n$ increasing to infinity, 
 with a (possibly nonlinear) difference operator as a Hamiltonian, acting on 
 a space of  dimension
higher than one, which can explain all these numerological coincidences?

\subsection*{Acknowledgments.}
This work began  in 2009, as part of an REU project
at the University of Michigan,  in
which the first author mentored the second author. 
The authors thank the University of Michigan
for support.

\end{document}